\documentclass[runningheads, envcountsame]{llncs}

\usepackage{amssymb,amsmath} 
\usepackage{graphicx}

\def\vertex{\circle*{1.8}}

\begin{document}

\title{The maximum weight stable set problem in
$(P_6,\mbox{bull})$-free graphs\thanks{The authors are partially
supported by ANR project STINT (reference ANR-13-BS02-0007).}}

\titlerunning{MWSS in $(P_6,\mbox{bull})$-free graphs}

\author{Fr\'ed\'eric Maffray\inst{1} \and Lucas Pastor\inst{2}}

\authorrunning{Fr\'ed\'eric Maffray and Lucas Pastor}

\institute{CNRS, Laboratoire G-SCOP, University of Grenoble-Alpes,
France  \and Laboratoire G-SCOP, University of Grenoble-Alpes,
France}

\maketitle

\begin{abstract}
We present a polynomial-time algorithm that finds a maximum weight
stable set in a graph that does not contain as an induced subgraph an
induced path on six vertices or a bull (the graph with vertices $a, b,
c, d, e$ and edges $ab, bc, cd, be, ce$).
    
\medskip

{\it Keywords}: Stability, $P_6$-free, bull-free, polynomial time,
algorithm
\end{abstract}

\date{\today}

\section{Introduction}

In a graph $G$, a \emph{stable set} (also called \emph{independent}
set) is any subset of pairwise non-adjacent vertices.  The {\sc
maximum stable set problem} (henceforth MSS) is the problem of finding
a stable set of maximum size.  In the weighted version of this
problem, each vertex $x$ of $G$ has a weight $w(x)$, and the weight of
any subset of vertices is defined as the total weight of its elements.
The {\sc maximum weight stable set problem} (MWSS) is then the
problem of finding a stable set of maximum weight.  It is well-known
that MSS (and consequently MWSS) is NP-hard in general, even under
various restrictions \cite{GarJoh}.

Given a fixed graph $F$, a graph $G$ \emph{contains} $F$ when $F$ is
isomorphic to an induced subgraph of $G$.  A graph $G$ is said to be
\emph{$F$-free} if it does not contain $F$.  Let us say that $F$ is
\emph{special} if every component of $F$ is a tree with no vertex of
degree at least four and with at most one vertex of degree three.
Alekseev \cite{Alex} proved that MSS remains NP-complete in the class
of $F$-free graphs whenever $F$ is not special.  On the other hand,
when $F$ is a special graph, it is still an open problem to know if
MSS can be solved in polynomial time in the class of $F$-free graphs
for most instances of $F$.  It is known that MWSS is polynomial-time
solvable in the class of $F$-free graphs when $F$ is any special graph
on at most five vertices \cite{alex2,LVV,LMo}.  Hence the new frontier
to explore now is the case where $F$ has six or more vertices.

We denote by $P_n$ the path on $n$ vertices.  The complexity
(polynomial or not) of MSS in the class of $P_6$-free graph is still
unknown, but it has recently been proved that it is quasi-polynomial
\cite{LPvL}.  There are several results on the existence of
polynomial-time algorithms for MSS in subclasses of $P_6$-free graphs;
see for example \cite{K,KM,RM1999,RM2009,RM2013}.  

The \emph{bull} is the graph with five vertices $a,b,c,d,e$ and edges
$ab$, $bc$, $cd$, $be$, $ce$ (see Figure~\ref{fig:bull}).  Our main
result is the following.

\begin{figure}[ht]
\unitlength=0.08cm
\thicklines
\begin{center}  
\begin{picture}(24,12) 
\multiput(6,0)(12,0){2}{\vertex}
\multiput(0,12)(24,0){2}{\vertex}
\put(12,9){\vertex}
\put(6,0){\line(1,0){12}}
\put(18,0){\line(1,2){6}} \put(6,0){\line(-1,2){6}}
\put(12,9){\line(2,-3){6}}\put(12,9){\line(-2,-3){6}}
\end{picture}
\end{center}
\caption{The bull.}
\label{fig:bull}
\end{figure}
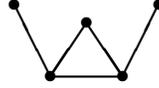

\begin{theorem}\label{thm:main}
MWSS can be solved in time $O(n^7)$ for every graph on $n$ vertices in
the class of $(P_6,\mbox{bull})$-free graphs.
\end{theorem}
Our proof of this theorem works along the following lines.  First, we
can reduce the problem to prime graphs, using modular decomposition
(all the technical terms will be defined precisely below).  Next, we
will show that if a prime $(P_6,\mbox{bull})$-free graph $G$ contains
a certain graph $G_7$, then $G$ has a structure from which we can
solve MWSS in polynomial time on $G$.  Finally, we will show that if a
prime $(P_6,\mbox{bull})$-free graph $G$ contains no $G_7$, then the
non-neighborhood of any vertex $x$ is perfect, which implies that a
maximum weight stable set containing $x$ can be found in polynomial
time, and it suffices to repeat this for every vertex $x$.

\medskip

Let us recall some definitions and results we need.  Let $G$ be a
graph.  For each vertex $v\in V(G)$, we denote by $N(v)$ the set of
vertices adjacent to $v$ (the \emph{neighbors} of $v$) in $G$.  For
any subset $S$ of $V(G)$ we write $N_S(v)$ instead of $N(v)\cap S$;
and for a subgraph $H$ we write $N_H(v)$ instead of $N_{V(H)}(v)$.  We
denote by $G[S]$ the induced subgraph of $G$ with vertex-set $S$, and
we denote by $N(S)$ the set $\{v\in V(G)\setminus S\mid$ $v$ has a
neighbor in $S\}$.  The complement of $G$ is denoted by
$\overline{G}$.  We say that a vertex $v$ is \emph{complete} to $S$ if
$v$ is adjacent to every vertex in $S$, and that $v$ is
\emph{anticomplete} to $S$ if $v$ has no neighbor in $S$.  For two
sets $S,T\subseteq V(G)$ we say that $S$ is complete to $T$ if every
vertex of $S$ is adjacent to every vertex of $T$, and we say that $S$
is anticomplete to $T$ if no vertex of $S$ is adjacent to any vertex
of $T$.

Let $\omega(G)$ denote the maximum size of a clique in $G$, and let
$\chi(G)$ denote the chromatic number of $G$ (the smallest number of
colors needed to color the vertices of $G$ in such a way that no two
adjacent vertices receive the same color).  A graph $G$ is
\emph{perfect} \cite{ber60,ber61,ber85} if every induced subgraph $H$
of $G$ satisfies $\chi(H)=\omega(H)$.  By the Strong Perfect Graph
Theorem \cite{CRST}, a graph is perfect if and only if $G$ and
$\overline{G}$ contain no induced $\ell$-cycle for any odd $\ell\ge
5$.
 
In a graph $G$ a \emph{homogeneous set} is a set $S\subseteq V(G)$
such that every vertex in $V(G)\setminus S$ is either complete to $S$
or anticomplete to $S$.  A homogeneous set is \emph{proper} if it
contains at leat two vertices and is different from $V(G)$.  A graph
is \emph{prime} if it has no proper homogeneous set.  Note that prime
graphs are connected.

A class of graphs is \emph{hereditary} if, for every graph $G$ in the
class, every induced subgraph of $G$ is also in the class.  For
example, for any family ${\cal F}$ of graphs, the class of ${\cal
F}$-free graphs is hereditary.  We will use the following theorem of
Lozin and Milani\v{c} \cite{LM}.
\begin{theorem}[\cite{LM}]\label{thm:LM}
Let $\cal{G}$ be a hereditary class of graphs.  Suppose that there is
a constant $c \geq 1$ such that the MWSS problem can be solved in time
$O(|V(G)|^c)$ for every prime graph $G$ in $\cal{G}$.  Then the MWSS
problem can be solved in time $O(|V(G)|^c + |E(G)|)$ for every graph
$G$ in $\cal{G}$.
\end{theorem}
Clearly, the class of $(P_6,\mbox{bull})$-free graphs is hereditary.
By Theorem~\ref{thm:LM}, in order to prove Theorem~\ref{thm:main} it
suffices to prove it for prime graphs.  This is the object of the
following theorem.

\begin{theorem}\label{thm:gmf}
Let $G$ be a prime $(P_6,\mbox{bull})$-free graph, and let $x$ be any
vertex in $G$.  Suppose that there is a $5$-cycle induced by
non-neighbors of $x$.  Then there is a clique $F$ (possibly empty) in $G$ such that the
induced subgraph $G\setminus F$ is triangle-free, and such a set $F$
can be found in time $O(n^2)$.
\end{theorem}
The proof of Theorem~\ref{thm:gmf} is given in the next section.  We
close this section by showing how to obtain a proof of
Theorem~\ref{thm:main} on the basis of Theorem~\ref{thm:gmf}.

\medskip

Our algorithm relies on results concerning graphs of bounded
clique-width.  We will not develop all the technical aspects concerning
the clique-width, but we recall its definition and the results that we
use.  The concept of clique-width was first introduced in~\cite{CER}.
The \emph{clique-width} of a graph $G$ is defined as the minimum
number of labels which are necessary to generate $G$ by using the
following operations:
\begin{itemize}
\item 
Create a vertex $v$ labeled by integer $i$.  
\item 
Make the disjoint union of two labeled graphs.  
\item 
Join all vertices with label $i$ to all vertices with label $j$ for
two labels $i \neq j$.  
\item 
Relabel all vertices of label $i$ by label $j$.
\end{itemize}
A \emph{$c$-expression} for a graph $G$ of clique-width $c$ is a
sequence of the above four operations that generate $G$ and use at
most $c$ different labels.  A class of graphs $\mathcal{C}$ has
\emph{bounded clique-width} if there exists a constant $c$ such that
every graph $G$ in $\mathcal{C}$ has clique-width at most $c$.

\begin{theorem}[\cite{CMR}]\label{thm:msol}
If a class of graphs $\mathcal{C}$ has bounded clique-width $c$, and
there is a function $f$ such that for every graph $G$ in $\mathcal{C}$
with $n$ vertices and $m$ edges a $c$-expression can be found in time
$O(f(n, m))$, then the maximum weight stable set problem can be
solved in time $O(f(n, m))$ for every graph $G$ in $\mathcal{C}$.
\end{theorem}

A \emph{triangle} is a complete graph on three vertices.
\begin{theorem}[\cite{BKM}]\label{thm:BKM}
The class of $(P_6, \mbox{triangle})$-free graphs has bounded
clique-width $c$, and a $c$-expression can be found in time
$O(|V(G)|^2)$ for every graph $G$ in this class.
\end{theorem}

Hence, as observed in \cite{BKM}, Theorems~\ref{thm:msol} 
and~\ref{thm:BKM} imply the following.
\begin{corollary}[\cite{BKM}]\label{cor:boundedcw}
For any $(P_6, \mbox{triangle})$-free graph $G$ on $n$ vertices one
can find a maximum weight stable set of $G$ in time $O(n^2)$.
\end{corollary}

A \emph{$k$-wheel} is a graph that consists of a $k$-cycle plus a
vertex (called the center) adjacent to all vertices of the cycle.  The
following lemma was proved for $k\ge 7$ in~\cite{RSb}; actually the
same proof holds for all $k\ge 6$ as observed in \cite{FMP}.
\begin{lemma}[\cite{RSb,FMP}]\label{lem:wheel}
Let $G$ be a bull-free graph.  If $G$ contains a $k$-wheel for any
$k\ge 6$, then $G$ has a proper homogeneous set.
\end{lemma}
Note that the bull is a self-complementary graph, so the preceding
lemma also says that if $G$ is prime then it does not contain the
complementary graph of a $k$-wheel with $k\ge 6$.

\begin{theorem}\label{thm:fin}
Let $G$ be a prime $(P_6,\mbox{bull})$-free graph on $n$ vertices.
Then a maximum weight stable set of $G$ can be found in time $O(n^7)$.
\end{theorem}
\begin{proof}
Let $G$ be a prime $(P_6,\mbox{bull})$-free graph.  Let $w : V(G)
\rightarrow \mathbb{N}$ be a weight function on the vertex set of $G$.
To find the maximum weight stable set in $G$ it is sufficient to
compute, for every vertex $x$ of $G$, a maximum weight stable set
containing $x$.  So let $x$ be any vertex in $G$.  We want to compute
the weight of a maximum stable set containing $x$.  Clearly it
suffices to compute the maximum weight stable set in each component of
the induced subgraph $G\setminus(\{x\}\cup N(x))$ and make the sum
over all components.  Let $K$ be any component of
$G\setminus(\{x\}\cup N(x))$.  We claim that:
\begin{equation}\label{perfectc5}
\mbox{Either $K$ is perfect or it contains a $5$-cycle.}
\end{equation}
Proof: Suppose that $K$ is not perfect.  Note that $K$ contains no odd
hole of length at least $7$ since $G$ is $P_6$-free.  By the Strong
Perfect Graph Theorem $K$ contains an odd antihole $C$.  If $C$ has
length at least $7$ then $V(C)\cup\{x\}$ induces a wheel in
$\overline{G}$, so $G$ has a proper homogeneous set by
Lemma~\ref{lem:wheel}, a contradiction because $G$ is prime.  So $C$
has length $5$, i.e., $C$ is a $5$-cycle.  So (\ref{perfectc5}) holds.

\medskip

We can test in time $O(n^5)$ if $K$ contains a $5$-cycle.  This leads
to the following two cases.

Suppose that $K$ contains no $5$-cycle. Then (\ref{perfectc5}) imples
that $K$ is perfect.  In that case we can use the algorithms from
either \cite{DM} or \cite{Penev}, which compute a maximal weight
stable set in a bull-free perfect graph in polynomial time.  The
algorithm from \cite{Penev} has time complexity $O(n^6)$.  

Now suppose that $K$ contains a $5$-cycle.  Then by
Theorem~\ref{thm:gmf} we can find in time $O(n^2)$ a clique $F$ such
that $G\setminus F$ is triangle-free.  Consider any stable set $S$ in
$K$.  If $S$ contains no vertex from $F$, then $S$ is in the subgraph
$G\setminus F$, which is triangle-free.  By
Corollary~\ref{cor:boundedcw} we can find a maximum weight stable set
$S_F$ in $G\setminus F$ in time $O(n^2)$.  If $S$ contains a vertex
$f$ from $F$, then $S\setminus f$ is in the subgraph $G \setminus
(\{f\}\cup N(f))$, which, since $F$ is a clique, is a subgraph of
$G\setminus F$ and consequently is also triangle-free.  By
Corollary~\ref{cor:boundedcw} we can find a maximum weight stable set
$S'_f$ in $G \setminus (\{f\}\cup N(f))$ in time $O(n^2)$.  Then we
set $S_f=S'_f\cup\{f\}$.  We do this for every vertex $f\in F$.  Now
we need only compare the set $S_F$ and the sets $S_f$ (for all $f\in
F$) and select the one with the largest weight.  This takes time
$O(n^3)$ for each component $K$ that contains a $5$-cycle.  Repeating
the above for every component takes time $O(n^6)$ since the components
are vertex-disjoint.  Repeating this for every vertex $x$, the total
complexity is $O(n^7)$.  \qed
\end{proof}

Now Theorem~\ref{thm:main} follows directly from
Theorems~\ref{thm:fin} and~\ref{thm:LM}.

\section{Proofs}

In a graph $G$, let $H$ be a subgraph of $G$.  For each $k>0$, a
\emph{$k$-neighbor} of $H$ is any vertex in $V(G)\setminus V(H)$ that
has exactly $k$ neighbors in $H$.

\begin{lemma}\label{lem:c5n}
Let $G$ be a bull-free graph.  Let $C$ be an induced $5$-cycle in $G$,
with vertices $c_1, \ldots, c_5$ and edges $c_ic_{i+1}$ for each $i$
modulo $5$.  Then:
\begin{itemize}
\item[(i)]
Every $2$-neighbor of $C$ is adjacent to $c_i$ and $c_{i+2}$ for some
$i$.
\item[(ii)]
Every $3$-neighbor of $C$ is adjacent to $c_i$, $c_{i+1}$ and
$c_{i+2}$ for some $i$.
\item[(iii)]
Every $5$-neighbor of $C$ is adjacent to every $k$-neighbor with
$k\in\{1,2\}$.
\item[(iv)]
If $C$ has a $4$-neighbor non-adjacent to $c_i$ for some $i$, then
every $1$-neighbor of $C$ is adjacent to $c_i$.
\item[(v)]
If a non-neighbor of $C$ is adjacent to a $k$-neighbor of $C$, then
$k\in\{1,2,5\}$.
\end{itemize}
\end{lemma}
\begin{proof}
If either (i) or (ii) fails, there is a vertex $x$ that is either a
$2$-neighbor adjacent to $c_i$ and $c_{i+1}$ or a $3$-neighbor
adjacent to $c_i$, $c_{i+1}$ and $c_{i+3}$ for some $i$, and then
$\{c_{i-1}, c_i, x, c_{i+1}, c_{i+2}\}$ induces a bull.

(iii) Let $u$ be a $5$-neighbor of $C$ and $x$ be a $k$-neighbor of
$C$ with $k\in\{1,2\}$.  So for some $i$ the vertex $x$ is adjacent to
$c_i$ and maybe to $c_{i+2}$.  Then $u$ is adjacent to $x$, for
otherwise $\{x, c_i, c_{i+1}, u, c_{i+3}\}$ induces a bull.

(iv) Let $f$ be a $4$-neighbor of $C$ non-adjacent to $c_i$.  Suppose
that there is a $1$-neighbor $x$ not adjacent to $c_i$.  So, up to
symmetry, $x$ is adjacent to $c_{i+1}$ or $c_{i+2}$.  Then $x$ is
adjacent to $f$, for otherwise $\{x, c_{i+1}, c_{i+2}, f, c_{i-1}\}$
induces a bull; but then $\{x, f, c_{i-2}, c_{i-1}, c_i\}$ induces a
bull.

(v) Let $z$ be a non-neighbor of $C$ that is adjacent to a
$k$-neighbor $x$ with $k\in\{3,4\}$.  So there is an integer $i$ such
that $x$ is adjacent to $c_i$ and $c_{i+1}$ and not adjacent to
$c_{i+2}$.  Then $\{z,x,c_i,c_{i+1},c_{i+2}\}$ induces a bull.  \qed
\end{proof}

An \emph{umbrella} is a graph that consists of a $5$-wheel plus a
vertex adjacent to the center of the $5$-wheel only.
\begin{lemma}\label{lem:umbr}
Let $G$ be a bull-free graph.  If $G$ contains an umbrella, then $G$
has a homogeneous set (that contains the $5$-cycle of the umbrella).
\end{lemma}
\begin{proof}
Let $C$ be the $5$-cycle of the umbrella, with vertices $c_1, \ldots,
c_5$ and edges $c_ic_{i+1}$ for all $i$ modulo~$5$.  Let $A$ be the
set of vertices that are complete to $C$, and let $Z$ be the set of
vertices that are anticomplete to $C$.  Let:
\begin{eqnarray*}
A' &=& \{a\in A\mid a \mbox{ has a neighbor in } Z\}. \\
A'' &=& \{a\in A\setminus A'\mid a \mbox{ has a non-neighbor in }
A'\}.
\end{eqnarray*}
By the hypothesis that $C$ is part of an umbrella, we have
$A'\neq\emptyset$.  Let $H$ be the component of $G\setminus (A'\cup
A'')$ that contains $V(C)$.  We claim that:
\begin{equation}\label{apvh}
\mbox{$A'\cup A''$ is complete to $V(H)$.}
\end{equation}
Proof: Pick any $b\in A'\cup A''$ and $u\in V(H)$, and let us prove
that $b$ is adjacent to~$u$.  We use the following notation.  If $b\in
A'$, then $b$ has a neighbor $z\in Z$.  If $b\in A''$, then $b$ has a
non-neighbor $a'\in A'$, and $a'$ has a neighbor $z\in Z$, and $b$ is
not adjacent to $z$, for otherwise we would have $b\in A'$.

By the definition of $H$, there is a shortest path
$u_0$-$\cdots$-$u_p$ in $H$ with $u_0\in V(C)$ and $u_p=u$, and $p\ge
0$.  We know that $b$ is adjacent to $u_0$ by the definition of~$A$.
First, we show that $b$ is adjacent to $u_1$ and finally 
by induction on $j = 2, \ldots p$, we show that $b$ is adjacent to $u_j$.
 
Now suppose that $p\ge 1$.  The vertex $u_1$ is a $k$-neighbor of $C$
for some $k\ge 1$.  If $k\in\{1,2\}$, then $b$ is adjacent to $u_1$ by
Lemma~\ref{lem:c5n}~(iii).  Suppose that $k\in\{3,4\}$.  Then there is
an integer $i$ such that $u_1$ is adjacent to $c_i$ and not to
$c_{i+1}$.  By Lemma~\ref{lem:c5n}~(v), $z$ is not adjacent to $u_1$.
If $b\in A'$, then $b$ is adjacent to $u_1$, for otherwise
$\{z,b,c_{i+1},c_{i},u_1\}$ induces a bull.  If $b\in A''$, then, by
the preceding sentence we know that $a'$ is adjacent to $u_1$; and
then $b$ is adjacent to $u_1$, for otherwise $\{z,a',u_1,u_0,b\}$
induces a bull.  Suppose that $k=5$.  So $u_1\in A$.  Then $u_1$ is
not adjacent to $z$, for otherwise we would have $u_1\in A'$.  If
$b\in A'$, then $b$ is adjacent to $u_1$ for otherwise we would have
$u_1\in A''$.  If $b\in A''$, then, by the preceding sentence we know
that $a'$ is adjacent to $u_1$; and then $b$ is adjacent to $u_1$, for
otherwise $\{z,a',u_1,u_0,b\}$ induces a bull.  

Finally suppose that $p\ge 2$.  So $u_2,\ldots,u_p$ are non-neighbors
of $C$.  Since $u_2\in Z$, we have $k\neq 5$, for otherwise we would
have $u_1\in A'$.  So there is an integer $h$ such that $u_1$ is
adjacent to $c_h$ and not to $c_{h+2}$.  We may assume up to
relabeling that $u_0=c_h$.  It follows that $c_{h+2}$ has no neighbor
in $\{u_0,\ldots,u_p\}$.  Then, by induction on $j=2,\ldots,p$, the
vertex $b$ is adjacent to $u_j$, for otherwise
$\{c_{h+2},b,u_{j-2},u_{j-1},u_j\}$ induces a bull.  So $b$ is
adjacent to $u$.  Thus (\ref{apvh}) holds.

\medskip

Let $R=V(G)\setminus (A'\cup A''\cup V(H))$.  By the definition of
$H$, there is no edge between $V(H)$ and $R$.  By~(\ref{apvh}), $V(H)$
is complete to $A'\cup A''$.  Hence $V(H)$ is a homogeneous set that
contains $V(C)$, and it is proper since $A'\neq \emptyset$.  \qed
\end{proof}

\begin{lemma}\label{lem:c5n2}
Let $G$ be a prime $(P_6,\mbox{bull})$-free graph.  Let $C$ be an
induced $5$-cycle in $G$.  If a non-neighbor of $C$ is adjacent to a
$k$-neighbor of $C$, then $k=2$.
\end{lemma}
\begin{proof}
Let $C$ have vertices $c_1, \ldots, c_5$ and edges $c_ic_{i+1}$ for
each $i$ modulo $5$.  Suppose that a non-neighbor $z$ of $C$ is
adjacent to a $k$-neighbor $x$ of $C$.  By Lemma~\ref{lem:c5n}~(v), we
have $k\in\{1,2,5\}$.  If $k=1$, say $x$ is adjacent to $c_i$, then
$z$-$x$-$c_i$-$c_{i+1}$-$c_{i+2}$-$c_{i+3}$ is an induced $P_6$ in
$G$.  If $k=5$, then $V(H)\cup\{x,y\}$ induces an umbrella, so, by
Lemma~\ref{lem:umbr}, $G$ has a proper homogeneous set, a
contradiction.  So $k=2$.  \qed
\end{proof}

Let $G_7$ be the graph with vertex-set $\{c_1, \ldots, c_5,d,x\}$ and
edge-set $\{c_ic_{i+1}\mid$ for all $i \bmod 5\}\cup\{dc_1,dc_4,dx\}$.
See Figure~\ref{fig:g7}.
\begin{figure}
\begin{center}
    \includegraphics{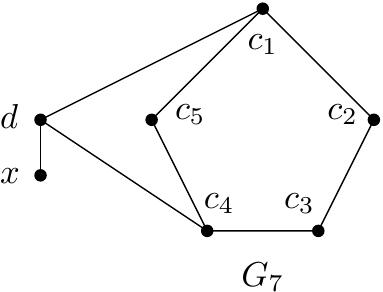}
\end{center}
\caption{The graph $G_7$.}\label{fig:g7}
\end{figure}

\begin{lemma}\label{lem:G7}
Let $G$ be a prime $(P_6,\mbox{bull})$-free graph.  Assume that $G$
contains a $5$-cycle $C$, with vertices $c_1, \ldots, c_5$ and edges
$c_ic_{i+1}$ for all $i\bmod 5$.  Moreover assume that $C$ has a
non-neighbor $x$ in $G$.  Then:
\begin{enumerate} 
\item[(i)]
There is a neighbor $d$ of $x$ that is a $2$-neighbor of $C$.  And
consequently, $V(C)\cup\{d,x\}$ induces a $G_7$.
\item[(ii)]
$C$ has no $3$-neighbor and no $5$-neighbor.
\item[(iii)]
If the vertex $d$ from (i) is (up to symmetry) adjacent to $c_1$ and
$c_4$, then every $4$-neighbor of $C$ is non-adjacent to $c_5$.
\end{enumerate}
\end{lemma}
\begin{proof}
Since $G$ is prime it is connected, so there is a shortest path from
$C$ to $x$ in $G$.  Let $x_0$-$\cdots$-$x_p$ be such a path, where
$x_0\in V(C)$ and $x_p=x$, and $p\ge 2$.  By Lemma~\ref{lem:c5n2},
$x_1$ is a $2$-neighbor of $C$, so up to relabeling we may assume that
$x_1$ is adjacent to $c_1$ and $c_4$.  Then $p=2$ for otherwise
$x_3$-$x_2$-$x_1$-$c_1$-$c_2$-$c_3$ is an induced $P_6$.  So (i) holds
with $d=x_1$.  Clearly, $\{c_1, \ldots, c_5, x_1, x\}$ induces a
$G_7$.

Therefore we may assume, up to symmetry, that the vertex $d$ from (i)
is adjacent to $c_1$ and $c_4$.

Suppose that there is a vertex $u$ that is either a $5$-neighbor of
$C$ or a $4$-neighbor adjacent to $c_5$.  In either case we may
assume, up to symmetry, that $u$ is adjacent to $c_1$, $c_3$ and
$c_5$.  Then $u$ is adjacent to $d$, for otherwise $\{d, c_1, c_5, u,
c_3\}$ induces a bull, and $u$ is adjacent to $x$, for otherwise $\{x,
d, c_1, u, c_3\}$ induces a bull.  But then $u$ and $x$ contradict
Lemma~\ref{lem:c5n2}.  This proves item (iii) and that $C$ has no
$5$-neighbor.

Finally suppose that $C$ has a $3$-neighbor $u$, adjacent to
$c_{i-1}$, $c_i$, $c_{i+1}$; we may assume up to symmetry that
$i\in\{5,1,2\}$.  Let $X$ be the set of vertices that are complete to
$\{c_{i-1}, c_{i+1}\}$ and anticomplete to $\{c_{i-2}, c_{i+2}\}$, and
let $Y$ be the vertex-set of the component of $G[X]$ that contains
$c_i$ and $u$.  Since $G$ is prime, $Y$ is not a homogeneous set, so
there is a vertex $t$ in $V(G)\setminus Y$ and vertices $y,z$ in $Y$
such that $t$ is adjacent to $y$ and not to $z$, and since $Y$ is
connected we may choose $y$ and $z$ adjacent.  We claim that:
\begin{equation}\label{thhh}
\mbox{$t$ is adjacent to $c_{i-2}$ and $c_{i+2}$ and to at least one
of $c_{i-1}$ and $c_{i+1}$.}
\end{equation}
Proof: If $t$ has no neighbor in $\{c_{i-1}, c_{i+1}\}$, then $t$ is
adjacent to $c_{i-2}$, for otherwise $\{t, y, z, c_{i-1}, c_{i-2}\}$
induces a bull, and similarly $t$ is adjacent to $c_{i+2}$; but then
$\{c_{i-1}, c_{i-2}, t, c_{i+2}, c_{i+1}\}$ induces a bull.  Hence $t$
has a neighbor in $\{c_{i-1}, c_{i+1}\}$.  Suppose that $t$ is
adjacent to both $c_{i-1}$ and $c_{i+1}$.  Since $t$ is not in $Y$ it
must have a neighbor in $\{c_{i-2}, c_{i+2}\}$, and actually $t$ is
complete to $\{c_{i-2}, c_{i+2}\}$, for otherwise $t$ is a
$3$-neighbor of the $5$-cycle induced by $\{z, c_{i-1}, c_{i-2},
c_{i+2}, c_{i+1}\}$ that violates Lemma~\ref{lem:c5n}~(ii).  Now
suppose that $t$ is adjacent to exactly one of $c_{i-1},c_{i+1}$, say
up to symmetry to $c_{i-1}$.  Then $t$ is adjacent to $c_{i-2}$, for
otherwise $\{c_{i-2}, c_{i-1}, t, y, c_{i+1}\}$ induces a bull, and
$t$ is adjacent to $c_{i+2}$, for otherwise $\{c_{i+2}, c_{i-2}, t,
c_{i-1}, z\}$ induces a bull.  Thus (\ref{thhh}) holds.

\medskip

Now we claim that:
\begin{equation}\label{xnoy}
\mbox{$x$ has no neighbor in $Y\cup\{t\}$.}
\end{equation}
Proof: Suppose that $x$ has a neighbor in $Y$.  Since $x$ also has a
non-neighbor $c_i$ in $Y$, and $Y$ is connected, there are adjacent
vertices $v,v'$ in $Y$ such that $x$ is adjacent to $v$ and not to
$v'$, and then $\{x, v, v', c_{i-1}, c_{i-2}\}$ induces a bull, a
contradiction.  So $x$ has no neighbor in $Y$.  In particular $x$ is
not adjacent to $y$, so $x$ has no neighbor in the $5$-cycle $C_y$
induced by $\{y, c_{i-1}, c_{i-2}, c_{i+2}, c_{i+1}\}$.  By
(\ref{thhh}), $t$ is a $3$- or $4$-neighbor of $C_y$.  By
Lemma~\ref{lem:c5n2}, $x$ is not adjacent to $t$.  Thus (\ref{xnoy})
holds.

\medskip

Suppose that $i=5$.  By~(\ref{thhh}), $t$ is adjacent to $c_2$ and
$c_3$ and, up to symmetry, to $c_1$.  Then $d$ is not adjacent to $y$,
for otherwise $\{x, d, y, c_1, c_2\}$ induces a bull, and $d$ is not
adjacent to $t$, for otherwise $\{x, d, c_1, t, c_3\}$ induces a bull;
but then $\{d, c_1, y, t, c_3\}$ induces a bull, a contradiction.

Suppose that $i=1$.  By~(\ref{thhh}), $t$ is adjacent to $c_3$ and
$c_4$.  Then $d$ is adjacent to $y$, for otherwise
$x$-$d$-$c_4$-$c_3$-$c_2$-$y$ is an induced $P_6$, and similarly $d$
is adjacent to $z$.  Then $t$ is adjacent to $d$, for otherwise
$\{x,d,z,y,t\}$ induces a bull, and $t$ is adjacent to $c_2$, for
otherwise $\{x,d,t,y,c_2\}$ induces a bull; but then
$\{x,d,c_4,t,c_2\}$ induces a bull.

Finally suppose that $i=2$.  By~(\ref{thhh}), $t$ is adjacent to $c_4$
and $c_5$.  Then $d$ is not adjacent to $y$, for otherwise
$\{x,d,c_1,y,c_3\}$ induces a bull, and $d$ is adjacent to $t$, for
otherwise $\{d,c_4,c_5,t,y\}$ induces a bull; but then
$\{x,d,c_4,t,y\}$ induces a bull, a contradiction.
\qed \end{proof}

\begin{theorem}\label{thm:structure}
Let $G$ be a prime $(P_6,\mbox{bull})$-free graph.  Suppose that $G$
contains a $G_7$, with vertex-set $\{c_1, \ldots, c_5,d,x\}$ and
edge-set $\{c_ic_{i+1}\mid$ for all $i \bmod 5\}\cup\{dc_1,dc_4,dx\}$.
Let:
\begin{itemize}
\item
$C$ be the $5$-cycle induced by $\{c_1, \ldots, c_5\}$;
\item
$F$ be the set of $4$-neighbors of $C$;
\item
$T$ be the set of $2$-neighbors of $C$;
\item
$W$ be the set of $1$-neighbors and non-neighbors of $C$.
\end{itemize}
Then the following properties hold:
\begin{itemize}
\item[(i)]
$V(G) = \{c_1,\ldots,c_5\}\cup F\cup T\cup W$.
\item[(ii)]
$F$ is complete to $\{c_1,\ldots,c_4\}$ and anticomplete to
$\{c_5,x,d\}$.
\item[(iii)]
$F$ is a clique.
\item[(iv)]
$G \setminus F$ is triangle-free.
\end{itemize}
\end{theorem}
\begin{proof}
Note that $d\in T$ and $x\in W$.  Clearly the sets $\{c_1, \ldots,
c_5\}$, $F$, $T$, and $W$ are pairwise disjoint subsets of $V(G)$.  We
observe that item~(i) follows directly from the definition of the sets
$F$, $T$, $W$ and Lemma~\ref{lem:G7}~(ii).

\medskip

Now we prove item (ii).  Consider any $f\in F$.  By
Lemma~\ref{lem:G7}~(iii), $f$ is non-adjacent to $c_5$, and
consequently $f$ is complete to $\{c_1, \ldots, c_4\}$.  Then $f$ is
not adjacent to $x$, for otherwise $\{x,f,c_3, c_4,c_5\}$ induces a
bull; and $f$ is not adjacent to $d$, for otherwise
$\{x,d,c_1,f,c_3\}$ induces a bull.  Thus (ii) holds.

\medskip

Now we prove item (iii).  Suppose on the contrary that $F$ is not a
clique.  So $G[F]$ has an anticomponent whose vertex-set $F'$
satisfies $|F'|\ge 2$.  Since $G$ is prime, $F'$ is not a homogeneous
set, so there are vertices $y,z\in F'$ and a vertex $t\in
V(G)\setminus F'$ that is adjacent to $y$ and not to $z$, and since
$F'$ is anticonnected we may choose $y$ and $z$ non-adjacent.  By the
definition of $F'$, we have $t\notin F$.  By~(ii), we have $t\notin
V(C)$.  Therefore, By~(i), we have $t\in T\cup W$. \\ 
Suppose that $t\in T$, so $t$ is adjacent to $c_{i-1}$ and $c_{i+1}$
for some $i$ in (up to symmetry) $\{1,2,5\}$.  If $i=1$, then
$\{z,c_2,y,t,c_5\}$ induces a bull.  If $i=2$, then $\{t, c_3,
z,c_4,c_5\}$ induces a bull.  So $i=5$.  Then $t$ is not adjacent to
$x$, for otherwise $\{x,t,c_1,y,c_3\}$ induces a bull.  Then $x$ is a
non-neighbor of the $5$-cycle induced by $\{c_1,c_2,c_3,c_4, t\}$, and
$y$ is a $5$-neighbor of that cycle, which contradicts
Lemma~\ref{lem:G7}.  \\
Hence $t\in W$.  By Lemma~\ref{lem:c5n}~(iv), $t$ is anticomplete to
$\{c_1,c_2,c_3,c_4\}$.  Then $t$ is adjacent to each $u\in\{c_5,d\}$,
for otherwise $\{t,y,c_3,c_4,u\}$ induces a bull.  So $t$ is a
$1$-neighbor of $C$, and by Lemma~\ref{lem:c5n2}, $t$ is not adjacent
to $x$.  But then $x$-$d$-$t$-$y$-$c_3$-$z$ is an induced $P_6$.  Thus
(iii) holds.

\medskip
 
There remains to prove item~(iv).  Suppose on the contrary that
$G\setminus F$ contains a triangle, with vertex-set $R=\{u,v,w\}$.
Clearly $C$ and $R$ have at most two common vertices.  Moreover:
\begin{equation}\label{nok3-2}
\mbox{$C$ and $R$ have at most one common vertex.}
\end{equation}
Proof: Suppose on the contrary that $u,v\in V(C)$, and consequently
$w\notin V(C)$.  By Lemma~\ref{lem:c5n}~(i), $w$ is a $k$-neighbor of
$C$ for some $k\ge 3$.  Since $w\notin F$, we have $k\neq 4$, so
$k\in\{3,5\}$; but this contradicts Lemma~\ref{lem:G7}~(ii).  So
(\ref{nok3-2}) holds.

\medskip

We observe that:
\begin{equation}\label{gmfc}
\mbox{$G\setminus F$ is connected.}
\end{equation}
Proof: Suppose that $G\setminus F$ is not connected.  The component of
$G\setminus F$ that contains $C$ also contains $T$.  Pick any vertex
$z$ in another component.  By Lemma~\ref{lem:G7}~(i), the vertex $z$
must have a neighbor in $T$, a contradiction.  So (\ref{gmfc}) holds.

\medskip

By~(\ref{gmfc}) there is a path from $C$ to $R$ in $G\setminus F$.
Let $P = p_0$-$\cdots$-$p_\ell$ be a shortest such path, with $p_0 \in
V(C)$, $p_\ell = u$, and $\ell \geq 0$.  Note that if $\ell\ge 1$, the
vertices $p_1, \ldots, p_\ell$ are not in $C$.  We choose $R$ so as to
minimize $\ell$.  Let $H$ be the component of $G[N(u)]$ that contains
$v$ and $w$.  Since $G$ is prime, $V(H)$ is not a homogeneous set, so
there are two vertices $y,z\in V(H)$ and a vertex $a\in V(G)\setminus
V(H)$ such that $a$ is adjacent to $y$ and not to $z$, and since $H$
is connected we may choose $y$ and $z$ adjacent.  By the definition of
$H$, the vertex $a$ is not adjacent to $u$.

\medskip

Suppose that $\ell=0$.  So $u=p_0=c_i$ for some $i\in\{1,\ldots,5\}$.
By~(\ref{nok3-2}) the vertices $y,z$ are not in $C$ and are
anticomplete to $\{c_{i-1},c_{i+1}\}$.  So, by
Lemma~\ref{lem:c5n}~(ii), each of $y$ and $z$ is a $1$- or
$2$-neighbor of $C$.  The vertex $a$ is adjacent to $c_{i-1}$, for
otherwise $\{a,y,z,c_i,c_{i-1}\}$ induces a bull; and similarly $a$ is
adjacent to $c_{i+1}$.  Note that this implies $a\notin V(C)$.
Suppose that $a$ has no neighbor in $\{c_{i-2},c_{i+2}\}$.  Then one
of $y,z$ has a neighbor in $\{c_{i-2},c_{i+2}\}$, for otherwise
$z$-$y$-$a$-$c_{i+1}$-$c_{i+2}$-$c_{i-2}$ is an induced $P_6$.  So
assume up to symmetry that one of $y,z$ is adjacent to $c_{i+2}$.
Then both $y,z$ are adjacent to $c_{i+2}$, for otherwise
$\{c_{i+2},y,z,c_i,c_{i-1}\}$ induces a bull.  So $y$ and $z$ are
$2$-neighbors of $C$, and they are not adjacent to $c_{i-2}$.  But
then $\{a,y,z,c_{i+2},c_{i-2}\}$ induces a bull, a contradiction.
Hence $a$ has a neighbor in $\{c_{i-2},c_{i+2}\}$.  By
Lemma~\ref{lem:c5n}~(ii) and Lemma~\ref{lem:G7}~(ii), $a$ must be adjacent to both
$c_{i-2},c_{i+2}$, so $a$ is a $4$-neighbor of $C$.  Hence $a\in F$,
and $i=5$, and by~(iii) $a$ has no neighbor in $\{d,x\}$.  The vertex
$z$ is not adjacent to $c_2$, for otherwise $\{z,c_2,c_1,a,c_4\}$
induces a bull; and similarly $z$ is not adjacent to $c_3$.  Then $y$
is not adjacent to $c_2$, for otherwise $\{c_4,c_5,z,y,c_2\}$ induces
a bull; and similarly $y$ is not adjacent to $c_3$.  So $y$ and $z$
are $1$-neighbors of $C$, and by Lemma~\ref{lem:c5n2} they are not
adjacent to $x$.  Then $d$ is adjacent to $y$, for otherwise
$\{d,c_1,c_2,a,y\}$ induces a bull, and $d$ is not adjacent to $z$,
for otherwise $\{x,d,z,y,a\}$ induces a bull; but then
$z$-$y$-$d$-$c_1$-$c_2$-$c_3$ is an induced $P_6$, a contradiction.
Therefore $\ell\ge 1$.

\medskip

We deduce that:
\begin{equation}\label{ci1}
\mbox{Every vertex $c_i$ in $C$ has at most one neighbor in 
$\{u,y,z\}$.}
\end{equation}
For otherwise, $c_i$ and two of its neighbors in $\{u,y,z\}$ form a
triangle that contradicts the choice of $R$ (the minimality of
$\ell$).  Thus (\ref{ci1}) holds.

\medskip

Suppose that $\ell\ge 2$.  By Lemma~\ref{lem:c5n2} (applied to $p_1$
and $p_2$), $p_1$ is a $2$-neighbor of $C$, adjacent to $c_{i-1}$ and
$c_{i+1}$ for some $i$.  The vertex $y$ has no neighbor $c_j$ in $C$,
for otherwise the path $c_j$-$y$ contradicts the choice of $P$.  The
vertex $p_2$ has no neighbor $c_j$ in $C$, for otherwise the path
$c_j$-$p_2$-$\cdots$-$p_\ell$ contradicts the choice of $P$.  Put
$p'=p_3$ if $\ell \ge 3$ and $p'=y$ if $\ell=2$.  Then
$p'$-$p_2$-$p_1$-$c_{i+1}$-$c_{i+2}$-$c_{i-2}$ is an induced $P_6$, a
contradiction.

Therefore $\ell=1$, so $u=p_1$.  By~(i), and since $u\notin F$, $u$ is
either a $1$-neighbor or a $2$-neighbor of $C$.

Suppose that $u$ is a $1$-neighbor of $C$, adjacent to $c_i$ for some
$i$.  By~(\ref{ci1}), $y$ and $z$ are not adjacent to $c_i$.  Then $a$
is adjacent to $c_i$, for otherwise $\{a,y,z,u,c_i\}$ induces a bull.
If $a$ has a neighbor in $\{c_{i-1}, c_{i+1}\}$, then, by
Lemma~\ref{lem:c5n}~(ii) and Lemma~\ref{lem:G7}~(ii), $a$ is a
$4$-neighbor of $C$; but then $a$ and $u$ violate
Lemma~\ref{lem:c5n}~(iv).  So $a$ has no neighbor in $\{c_{i-1},
c_{i+1}\}$.  Then $z$ is not adjacent to $c_{i+1}$, for otherwise, by
(\ref{ci1}), $\{a,y,u,z,c_{i+1}\}$ induces a bull; and $z$ has no
neighbor $c$ in $\{c_{i-2},c_{i+2}\}$, for otherwise, by (\ref{ci1}),
$\{c_i, u, y, z, c\}$ induces a bull.  But then
$z$-$u$-$c_i$-$c_{i+1}$-$c_{i+2}$-$c_{i-2}$ is an induced $P_6$, a
contradiction.

Therefore $u$ is a $2$-neighbor of $C$, adjacent to $c_{i-1}$ and
$c_{i+1}$ for some $i$.  By~(\ref{ci1}), $y$ and $z$ are anticomplete
to $\{c_{i-1},c_{i+1}\}$.  The vertex $c_{i+2}$ has no neighbor in
$\{y,z\}$, for otherwise, by (\ref{ci1}), $\{c_{i+2},y,z,u,c_{i-1}\}$
induces a bull.  Likewise, $c_{i-2}$ has no neighbor in $\{y,z\}$.
The vertex $a$ is adjacent to $c_{i-1}$, for otherwise
$\{a,y,z,u,c_{i-1}\}$ induces a bull, and similarly $a$ is adjacent to
$c_{i+1}$.  Then $a$ has a neighbor in $\{c_{i-2},c_{i+2}\}$, for
otherwise $z$-$y$-$a$-$c_{i+1}$-$c_{i+2}$-$c_{i-2}$ is an induced
$P_6$.  By Lemma~\ref{lem:c5n}~(ii) and Lemma~\ref{lem:G7}~(ii), $a$
is a $4$-neighbor of $C$, so $i=5$, and $a$ has no neighbor in
$\{c_5,d,x\}$.  Then $y$ is adjacent to $c_5$, for otherwise
$\{y,a,c_3,c_4,c_5\}$ induces a bull; and by (\ref{ci1}), $z$ is not adjacent to
$c_5$. But then $z$-$y$-$c_5$-$c_4$-$c_3$-$c_2$ is an induced $P_6$, a
contradiction.  This completes the proof of the theorem.  \qed
\end{proof}
Finally, Theorem~\ref{thm:gmf} follows as a direct consequence of
Lemma~\ref{lem:G7} and Theorem~\ref{thm:structure}.

\section{Conclusion}

In a parallel paper \cite{MP1}, but using other techniques, we proved
that the problem of deciding if a $(P_6,\mbox{bull})$-free graph is
$4$-colorable can be solved in polynomial time.  It is not known if
there exists a polynomial-time algorithm that determines
$4$-colorability in the whole class of $P_6$-free graphs.

We note that the whole class of $(P_6,\mbox{bull})$-free graph does
not have bounded clique-width, since it contains the class of
complements of bipartite graphs, which has unbounded clique-width
\cite{GR}.  Hence the main result of this paper and of \cite {MP1}
cannot be obtained solely with a bounded clique-width argument.


\end{document}